\documentclass[a4paper,leqno,12pt]{amsart}
\usepackage[utf8]{inputenc}

\usepackage[top=25mm, bottom=25mm, left=20mm, right=20mm]{geometry}
\usepackage{amsmath, amstext, amssymb, enumitem, mathtools}
\usepackage{caption}

\usepackage[hidelinks]{hyperref}
\usepackage{nameref}

\hypersetup{colorlinks = true, urlcolor = blue, linkcolor = blue, citecolor = red}

\newcommand{\HH}{\mathbb{H}}

\newcommand{\RR}{\mathbb{R}}
\newcommand{\HS}{\mathbb{V}^{d+1}_0}
\newcommand{\HI}{\mathbb{V}^{d+1}}

\newcommand{\me}{\mathrm{e}}
\newcommand{\dint}{\, \text{\normalfont{d}}}
\newcommand{\sign}{{\rm sign}}
\newcommand{\measure}[2]{\dint \Pi_{#1 - \frac{1}{2}}(#2)}

\DeclarePairedDelimiter{\norm}{\lVert}{\rVert}
\DeclarePairedDelimiter{\abs}{\lvert}{\rvert}
\DeclarePairedDelimiter{\inner}{\langle}{\rangle}

\numberwithin{equation}{section}

\theoremstyle{plain}
\newtheorem{theorem}[equation]{Theorem}
\newtheorem{lemma}[equation]{Lemma}
\newtheorem{corollary}[equation]{Corollary}
\newtheorem{definition}[equation]{Definition}
\newtheorem{observation}[equation]{Observation}

\theoremstyle{definition}
\newtheorem{remark}[equation]{Remark}

\begin{document}

	\title[Sharp estimates for Jacobi heat kernels in double conic domains]{Sharp estimates for Jacobi heat kernels \\ in double conic~domains}

	\author{Dawid Hanrahan}
	\address{Dawid Hanrahan (\textnormal{dawid.hanrahan@pwr.edu.pl})
		\newline WUST -- Wroc\l aw University of Science and Technology, 50-370 Wroc\l aw, Poland}

	\begin{abstract} We study the even and odd Jacobi heat kernels defined in the context of the multidimensional double cone $\HI$ and its surface $\HS$, the multidimensional hyperboloid $\prescript{}{\varrho}{\HI}$ and its surface $\prescript{}{\varrho}{\HS}$, and the multidimensional paraboloid  $\widetilde{\mathbb{V}}^{d+1}$ and its surface $\widetilde{\mathbb{V}}^{d+1}_0$. By integrating the framework of Jacobi polynomials on these domains, as analyzed by Xu, with contemporary methods developed by Nowak, Sjögren, and Szarek for obtaining sharp estimates of the spherical heat kernel, we establish genuinely sharp estimates for the even and odd Jacobi heat kernels on $\HI$ and $\HS$, and for the even Jacobi heat kernel on $\prescript{}{\varrho}{\HI}$ and $\prescript{}{\varrho}{\HS}$. We also discuss the limitations encountered in extending these results to the remaining settings.

	\smallskip
	\noindent \textbf{2020 Mathematics Subject Classification:} Primary 35K08, 33C50.

	\smallskip
	\noindent \textbf{Key words:} multidimensional double cone, Jacobi heat kernel.
	\end{abstract}

\maketitle

\section{Introduction}

Heat kernels are fundamental in both mathematics and physics. Despite extensive research on this topic, obtaining genuinely sharp estimates has only recently extended beyond classical contexts like the hyperbolic space $\HH^{d+1}$, see \cite{DM88}. Sharp estimates for the spherical heat kernel were achieved in \cite{NSS18}, and later extended to Jacobi heat kernels for all compact rank-one symmetric spaces, including the classical domain $[-1,1]$, in \cite{NSS21} using earlier techniques from \cite{NS13}.

The goal of this paper is to establish genuinely sharp estimates for even and odd Jacobi heat kernels on various multidimensional conic and hyperbolic domains. We analyze the double cone $\HI$ and its surface $\HS$ along with the hyperboloid $\prescript{}{\varrho}{\HI}$ and its surface $\prescript{}{\varrho}{\HS}$, originally introduced by Xu in \cite{Xu21}, as well as the paraboloid $\widetilde{\mathbb{V}}^{d+1}$ and its surface $\widetilde{\mathbb{V}}^{d+1}_0$ introduced by Xu in \cite{Xu23}. The obtained results are stated in Theorems~\ref{T1} and \ref{T2} and Corollaries~\ref{C11} and \ref{C12} for double cones, in Corollaries~\ref{C1} and \ref{C2} for the hyperbolic settings, and in Observations \ref{O1} and \ref{O2} for the parabolic settings. The approach presented in this paper builds upon the methods developed in \cite{NSS21}. We define ``genuinely sharp'' as finding exact bounds for these kernels from both above and below, a level of accuracy previously attained only in a few specific settings \cite{BM16, MSZ16, MS20, Se22, HK23, NSS24}.

Before presenting the results, we briefly discuss the general framework of heat kernels, referencing \cite[Section 3]{DX14} and \cite[Sections~1~and~2]{Xu21}.

Given a domain $\Omega \subset \RR^d$ with an appropriate weight $\varpi$, one can construct an orthogonal polynomial basis in $\mathcal{L}^2(\Omega, \varpi)$. For each degree $n$, the subspace $\mathcal{V}_n(\varpi)$ spans the corresponding orthogonal polynomials. The orthogonal projection ${\rm proj}_n \colon \mathcal{L}^2(\Omega, \varpi) \to \mathcal{V}_n(\varpi)$ takes the form
\begin{equation*}
    {\rm proj}_n f(x) \coloneqq \int_\Omega f(y) P_n(\varpi; x,y) \dint \varpi(y),
\end{equation*}
where $P_n(\varpi; x,y)$ is the reproducing kernel for $\mathcal{V}_n(\varpi)$. Our focus is on $(\Omega, \varpi)$ satisfying two key properties.
\begin{enumerate} \itemsep=0.1cm
    \item \label{item1} There exists a linear second-order differential operator $\mathcal{D}$ for which the orthogonal polynomials are its eigenfunctions, and the eigenvalues depend only on the degree $n$. This diffusion operator, defined on a suitable subspace of $\mathcal{L}^2(\Omega, \varpi)$, is characterized by orthogonal~expansions.
    \item \label{item2} Each $P_n$ is given by a closed-form expression.
\end{enumerate}

If these properties hold, one can define the associated heat kernel for $\tau \in (0, \infty)$ as
\begin{equation*}
    h_{\tau}(\varpi;x,y) \coloneqq \sum_{n=0}^{\infty} \me^{- \tau \lambda_n^{\mathcal D}} P_{n}(\varpi;x,y),
\end{equation*}
where $\lambda_n^{\mathcal{D}}$ are eigenvalues of $\mathcal{D}$ corresponding to $\mathcal{V}_n(\varpi)$. Intuitively, $h_\tau(\varpi; x,y)$ captures the heat distribution between points $x$ and $y$ over time $\tau$ when $\mathcal{D}$ governs the diffusion process. Thus, it serves as a fundamental solution to initial value problems for the heat equation. When closed-form expressions for $P_n$ are available, sharp estimates for $h_{\tau}(\varpi;x,y)$ can be obtained for specific inputs.

Such scenarios, however, are uncommon. In \cite{Xu21}, Xu examined the double cone $\HI$ and its surface $\HS$, as well as the hyperboloid $\prescript{}{\varrho}{\HI}$ and its surface $\prescript{}{\varrho}{\HS}$, for each $d \geq 2$ and for even or odd orthogonal polynomials, separately. In most cases, Xu found weights $\varpi$ that allowed both properties to hold and derived exact formulas for the resulting Jacobi heat kernels, named after their clear connection to classical Jacobi theory. On the other hand, for the multidimensional paraboloid $\widetilde{\mathbb{V}}^{d+1}$ and its surface $\widetilde{\mathbb{V}}^{d+1}_0$, weights found by Xu in \cite{Xu23} do not allow both properties to hold (see Observations~\ref{O1} and \ref{O2}).

The following sections summarize Xu's contributions, focusing on essential aspects relevant to our results. Deriving the weight $\varpi$ involved additional tools such as spherical harmonics and classical Jacobi polynomials, reflecting how multidimensional double cones and hyperboloids inherit features of both intervals and Euclidean balls. Xu began with simpler settings (intervals and balls) before constructing orthogonal polynomials or diffusion operators on $\HI$, $\HS$, $\prescript{}{\varrho}{\HI}$, $\prescript{}{\varrho}{\HS}$, $\widetilde{\mathbb{V}}^{d+1}$, and $\widetilde{\mathbb{V}}^{d+1}_0$. Below we avoid restating all details, concentrating instead on key formulas critical to our analysis. For a comprehensive exposition, see \cite{Xu21} and \cite{Xu23}. Additionally you may check the monographs~\cite{DX13, DX14}.

\subsection{Jacobi Heat Kernel on the Surface of the Double Cone}

The content of this section is derived from \cite[Sections~3--5]{Xu21}. Fix $d \in \{2, 3, \ldots\}$ and consider the domain
\begin{equation}
    \label{def1}
    \HS \coloneqq \big\{(x, t) \in \mathbb{R}^{d} \times \mathbb{R} :  \norm{x}^2 = t^2, \, 0 \le \abs{t} \le 1 \big\}.
\end{equation}

We equip \eqref{def1} with the weight function
\begin{equation*}
    w_{\beta, \gamma}(t) = \abs{t}^{2\beta}(1 - t^2)^{\gamma - \frac{1}{2}}, \qquad \beta > -\frac{d}{2}, \quad \gamma > -\frac{1}{2},
\end{equation*}
which is exactly the range in which $w_{\beta, \gamma}$ is integrable with respect to the surface measure on $\HS$ (in \cite{Xu21} the wider range $\beta > -\frac{d+1}{2}$ is admitted, which suffices for the odd subspaces, cf. \cite[Remark~4.1]{Xu21}).

For each $n \in \{0, 1, 2, \dots \}$, the space $\mathcal{V}_n(w_{\beta, \gamma})$ of orthogonal polynomials of degree $n$ associated with $w_{\beta, \gamma}$ is defined in terms of the so-called Jacobi polynomials on the surface of the double~cone.

The space $\mathcal{V}_n(w_{\beta, \gamma})$ can be represented as a direct sum
\begin{equation*}
    \mathcal{V}_n(w_{\beta, \gamma}) = \mathcal{V}^E_n(w_{\beta, \gamma}) \bigoplus \mathcal{V}^O_n(w_{\beta, \gamma}),
\end{equation*}
where $\mathcal{V}^E_n(w_{\beta, \gamma})$ is the subspace consisting of polynomials that are even in the variable $t$ and, similarly, $\mathcal{V}^O_n(w_{\beta, \gamma})$ denotes the subspace consisting of polynomials that are odd in $t$.

Moreover, there exists a suitable second-order differential operator $\mathcal{D}^E_\gamma$ acting on a subspace of $\mathcal L^2(\HS, c_{\beta=0, \gamma} w_{\beta=0, \gamma})$, where $c_{\beta=0, \gamma}$ is the corresponding normalizing constant\footnote{That is, $c_{\beta=0, \gamma}$ is the unique constant such that $c_{\beta=0, \gamma} w_{\beta=0, \gamma}$ becomes a probability measure on $\HS$.}, with all $\mathcal{V}^E_n(w_{\beta=0, \gamma})$ being its eigenspaces. The associated eigenvalues are $-n(n + 2\gamma + d - 1)$.

We do not use the formula for $\mathcal{D}^E_\gamma$ in subsequent sections, but for completeness we recall~that
\begin{equation*}
\mathcal{D}^E_\gamma \coloneqq (1 - t^2) \partial^2_t - (2\gamma + d)t \partial_t +\frac{d-1}{t} \partial_t + \frac{1}{t^2} \Delta^{(x)}_0,
\end{equation*}
where $\Delta^{(x)}_0$ denotes the Laplace--Beltrami operator acting in the direction of the spherical variable $x$, see \cite[Theorem~4.2]{Xu21}.

Analogically there exists a second-order differential operator $\mathcal{D}^O_\gamma$ acting on a subspace of $\mathcal L^2(\mathbb{V}^{d+1}_0, c_{\beta=-1, \gamma} w_{\beta=-1, \gamma})$, where again $c_{\beta=-1, \gamma}$ is the corresponding normalizing constant, with all $\mathcal{V}^O_n(w_{\beta=-1, \gamma})$ being its eigenspaces. The associated eigenvalues are $-n(n + 2\gamma + d - 3)$. Since the surface measure on $\HS$ carries the factor $\abs{t}^{d-1}$, the weight $w_{\beta=-1, \gamma}$ is integrable over $\HS$ (so that $c_{\beta=-1, \gamma}$ is well defined) only if $d \ge 3$, in accordance with the range $\beta > -\frac{d}{2}$ above; accordingly, whenever $\beta = -1$ we assume $d \ge 3$.

Although we do not use the formula for $\mathcal{D}^O_\gamma$ later on, we recall that
\begin{equation*}
    \mathcal{D}^O_\gamma \coloneqq (1 - t^2) \partial^2_t - (2\gamma + d - 2)t \partial_t + \frac{d-3}{t} \bigg(\partial_t - \frac{1}{t}\bigg) + \frac{1}{t^2} \Delta^{(x)}_0,
\end{equation*}
see \cite[Theorem~4.2]{Xu21}.

If $\beta = 0$ and $\gamma \geq 0$, then the reproducing kernel of $\mathcal{V}^E_n(w_{\beta=0, \gamma})$ is given by the following formula (cf. \cite[(5.7)]{Xu21})
\begin{equation}
    \label{RK1}
    P^E_n\big(w_{\beta=0, \gamma}; (x, t), (y, s)\big) \coloneqq \int \limits_{-1}^{1} Z_{n}^{\gamma + \frac{d-1}{2}} \big(\xi(x, t, y, s; v)\big) \dint\Pi_{\gamma - \frac{1}{2}}(v).
\end{equation}
Here $\dint\Pi_{\eta}(z) \coloneqq c_{\eta}(1 - z^2)^{\eta - \frac{1}{2}} \dint z$ for $z \in [-1, 1]$ if $\eta \in (-\frac{1}{2}, \infty)$, where $c_{\eta}$ is the normalizing constant, while $\dint\Pi_{-\frac{1}{2}}$ is the mean of Dirac deltas $\frac{1}{2}(\delta_{-1} + \delta_{1})$. Furthermore,
\begin{equation*}
    \xi(x, t, y, s; v) \coloneqq \xi(v) \coloneqq I_1 + vI_2
\end{equation*}
with $I_1 \coloneqq \inner{x, y}$ and $I_2 \coloneqq \sqrt{1 - s^2}\sqrt{1 - t^2}$.\footnote{In \cite[(5.7)]{Xu21} (and in \cite[Corollary~5.6]{Xu21}) the first term in the argument of $Z_n$ is printed as $\inner{x, y}\sign(st)$. This cannot be correct as stated: since $Z_n^{\lambda}(-z) = (-1)^n Z_n^{\lambda}(z)$ and $\dint \Pi_\eta$ is symmetric, the resulting expression is odd in $s$ for odd $n$, whereas the reproducing kernel of a space of polynomials even in $t$ is even in $s$. For $st > 0$ both forms coincide, and the sign-free form extends to $st < 0$ by evenness. Accordingly, we use $I_1 = \inner{x, y}$ (and, in Subsection~\ref{sec:solid}, $I_3$ without the sign factor).} In \eqref{RK1}, for $n \in \{0, 1, 2, \ldots\}$, $\lambda \in [0, \infty)$, and $z \in [-1, 1]$, we use the special function
\begin{equation}
    \label{Gegen}
    Z_{n}^{\lambda + \frac{1}{2}}(z) \coloneqq \frac{C_{n}^{\lambda + \frac{1}{2}}(1) C_{n}^{\lambda + \frac{1}{2}}(z)}{h^{\lambda + \frac{1}{2}}_n} = \frac{P_{n}^{\lambda, \lambda}(1) P_{n}^{\lambda, \lambda}(z)}{h_{n}^{\lambda, \lambda}},
\end{equation}
where $C_{n}^{\lambda + \frac{1}{2}}$ is the Gegenbauer polynomial of degree $n$, and $P_n^{\lambda, \lambda}$ is the classical Jacobi polynomial on the interval $[-1,1]$, while $h^{\lambda + \frac{1}{2}}_n$ and $h_n^{\lambda, \lambda}$ are the squares of their norms in the space $\mathcal L^2\big([-1, 1], \dint \Pi_{\lambda + \frac{1}{2}}\big)$. In \cite[Subsection~2.2]{Xu21}, the function $Z_{n}^{\lambda + \frac{1}{2}}$ is defined through the Gegenbauer polynomial, but here we will use the latter expression in \eqref{Gegen} instead. Both expressions are equivalent since $P_n^{\lambda, \lambda}$ is a constant multiple of $C_{n}^{\lambda + \frac{1}{2}}$ (see \cite[(4.7.1)]{Sz75}).

The relationship between $P^O_n$ and $P^E_n$ is given by the following formula
\begin{equation}
    \label{kernel-relation-surface} P^O_n\big(w_{\beta, \gamma}; (x, t), (y, s)\big) = \frac{\beta + \frac{d-1}{2} + \gamma + 1}{\beta + \frac{d}{2}} st P^E_{n-1}\big(w_{\beta+1, \gamma}; (x, t), (y, s)\big),
\end{equation}
for $n \in \{1, 2, \ldots\}$, $\beta > -\frac{d}{2}$ and $\gamma > -\frac{1}{2}$, see \cite[Theorem~5.1]{Xu21}\footnote{In \cite[Theorem~5.1]{Xu21} the author considered only $\beta > - \tfrac{1}{2}$, but the proof applies verbatim as long as the normalizing constants $c_{\beta, \gamma}$ and $c_{\beta + 1, \gamma}$ are finite, i.e. for $\beta > -\tfrac{d}{2}$; for $\beta = -1$ this means $d \ge 3$.}. Indeed, since $\abs{t}^{2\beta} t^2 = \abs{t}^{2\beta + 2}$, we have $\mathcal{V}^O_n(w_{\beta, \gamma}) = t \, \mathcal{V}^E_{n-1}(w_{\beta+1, \gamma})$, and the constant in \eqref{kernel-relation-surface} equals $c_{\beta+1, \gamma} / c_{\beta, \gamma}$.

In this paper, we focus on finding sharp estimates for even and odd Jacobi heat kernels separately; however, it is sufficient to find the sharp estimates for the even kernels. The estimates for the odd kernels are consequences of \eqref{kernel-relation-surface}, see Lemma~\ref{odd-even}.

\begin{definition}
	\label{HK1}
	Let $\gamma \in (-\frac{1}{2}, \infty)$. Then for each $\tau \in (0, \infty)$, the associated \emph{even Jacobi heat kernel} on $\HS$ is given by
    \begin{equation*}
        h^E_\tau\big(w_{\beta=0, \gamma}; (x, t), (y, s)\big) \coloneqq \sum \limits_{n=0}^{\infty} \me^{-\tau n (n + 2\gamma + d - 1)} P^E_n\big(w_{\beta=0, \gamma}; (x, t), (y, s)\big),
    \end{equation*}
    and, if $d \ge 3$, the associated \emph{odd Jacobi heat kernel} on $\HS$ is given by
    \begin{equation*}
        h^O_\tau\big(w_{\beta=-1, \gamma}; (x, t), (y, s)\big) \coloneqq \sum \limits_{n=1}^{\infty} \me^{-\tau n (n + 2\gamma + d - 3)} P^O_n\big(w_{\beta=-1, \gamma}; (x, t), (y, s)\big).
    \end{equation*}
\end{definition}

\noindent According to this, for the even Jacobi heat kernel on $\HS$ we shall show the following result.

\begin{theorem} \label{T1}
	Let $\gamma \in [0, \infty)$. Then $h^E_\tau\big(w_{\beta=0, \gamma}; (x, t), (y, s)\big)$, the even Jacobi heat kernel on $\HS$, is comparable to
    \begin{equation*}
        \tau^{-\frac{d}{2}} \big( \pi - \arccos(I_1 + I_2) + \tau \big)^{-\gamma - \frac{d}{2} + \frac{1}{2}} \bigg( \tau + \frac{I_2}{\pi - \arccos(I_1 + I_2)} \bigg)^{-\gamma} \exp \bigg\{-\frac{\arccos^2(
        I_1 + I_2)}{4\tau}\bigg\}
    \end{equation*}
    for $\tau \in (0, 1]$, and is comparable to 1 for $\tau \in (1, \infty)$, uniformly in $(x, t), (y, s) \in \HS$. Here and below a quotient with vanishing numerator is understood as $0$; note that $I_1 + I_2 = -1$ forces $I_2 = 0$, see Remark~\ref{Remark: range-xi}.
\end{theorem}

A consequence of \eqref{kernel-relation-surface} and Theorem \ref{T1} is the following (see Lemma~\ref{odd-even} and Section~\ref{sec:proofs}).

\begin{corollary}
    \label{C11}
    Let $\gamma \in [0, \infty)$ and $d \ge 3$. Then $h^O_\tau\big(w_{\beta=-1, \gamma}; (x, t), (y, s)\big)$, the odd Jacobi heat kernel on $\HS$, is of the form $st \, k_\tau\big((x, t), (y, s)\big)$, where $k_\tau\big((x, t), (y, s)\big)$ is positive and comparable to
    \begin{equation*}
        \tau^{-\frac{d}{2}} \big( \pi - \arccos(I_1 + I_2) + \tau \big)^{-\gamma - \frac{d}{2} + \frac{1}{2}} \bigg( \tau + \frac{I_2}{\pi - \arccos(I_1 + I_2)} \bigg)^{-\gamma} \exp \bigg\{-\frac{\arccos^2(
        I_1 + I_2)}{4\tau}\bigg\}
    \end{equation*}
    for $\tau \in (0, 1]$, and to $\me^{-\tau(2\gamma + d - 2)}$ for $\tau \in (1, \infty)$, uniformly in $(x, t), (y, s) \in \HS$. In fact, $k_\tau = c \, \me^{-\tau(2\gamma + d - 2)} h^E_\tau(w_{\beta=0, \gamma}; \cdot, \cdot)$ with the constant $c = c_{d, \gamma} > 0$ from Lemma~\ref{odd-even}.
\end{corollary}

\subsection{Jacobi Heat Kernel on the Solid Double Cone} \label{sec:solid}

The content of this section is, again, derived from \cite[Sections~3--5]{Xu21}. Fix $d \in \{2, 3, \ldots\}$ and consider the domain
\begin{equation}
    \label{def2}
    \HI \coloneqq \big\{(x, t) \in \mathbb{R}^{d} \times \mathbb{R} : \norm{x}^2 \le   t^2, \, 0 \le \abs{t} \le 1 \big\}.
\end{equation}

We equip \eqref{def2} with the weight function
\begin{equation*}
    w_{\beta, \gamma, \mu}(t) = \abs{t}^{2\beta}(1 - t^2)^{\gamma - \frac{1}{2}}(t^2 - \norm{x}^2)^{\mu - \frac{1}{2}}, \qquad \beta > -\mu - \frac{d}{2}, \quad \gamma > -\frac{1}{2}, \quad \mu > -\frac{1}{2},
\end{equation*}
which is exactly the range in which $w_{\beta, \gamma, \mu}$ is integrable over $\HI$ (integration over the ball $\norm{x} \le \abs{t}$ produces the factor $\abs{t}^{2\mu + d - 1}$).

As before, for each $n \in \{0, 1, 2, \dots \}$, the space $\mathcal{V}_n(w_{\beta, \gamma, \mu})$ of orthogonal polynomials of degree $n$ associated with $w_{\beta, \gamma, \mu}$ is defined in terms of the so-called Jacobi polynomials on the solid double~cone and can be represented as a direct sum of $\mathcal{V}^E_n(w_{\beta, \gamma, \mu})$ and $\mathcal{V}^O_n(w_{\beta, \gamma, \mu})$ (subspaces consisting of polynomials, respectively, even and odd in the variable $t$).

Similarly, there exists a second-order differential operator $\mathcal{D}^E_{\gamma, \mu}$ acting on a subspace of $\mathcal L^2\big(\HS, c_{\beta=\frac{1}{2}, \gamma, \mu} w_{\beta=\frac{1}{2}, \gamma, \mu}\big)$, where $c_{\beta=\frac{1}{2}, \gamma, \mu}$ is the corresponding normalizing constant, with all $\mathcal{V}^E_n\big(w_{\beta=\frac{1}{2}, \gamma, \mu}\big)$ being its eigenspaces. The associated eigenvalues are $-n(n + 2\gamma + 2\mu + d)$.

Again, we do not use the formula for $\mathcal{D}^E_{\gamma, \mu}$, but for completeness we recall~that
\begin{align*}
\mathcal{D}^E_{\gamma, \mu} &\coloneqq (1 - t^2) \partial^2_t + \Delta_x - \inner{x, \nabla_x}^2 + \frac{2}{t}(1 - t^2) \inner{x, \nabla_x} \partial_t \\
&\quad + (2\mu + d) \frac{1}{t} \partial_t - t \partial_t - (2\gamma + 2\mu + d)(t \partial_t + \inner{x, \nabla_x}),
\end{align*}
where $\Delta_x$ denotes the Laplace operator and $\nabla_x$ is the gradient operator, both acting in the direction of the variable $x$, see \cite[Theorem~4.6]{Xu21}.

Analogously, there exists a second-order differential operator $\mathcal{D}^O_{\gamma, \mu}$ acting on a subspace of $\mathcal L^2\big(\HI, c_{\beta=-\frac{1}{2}, \gamma, \mu} w_{\beta=-\frac{1}{2}, \gamma, \mu}\big)$ with all $\mathcal{V}^O_n\big(w_{\beta=-\frac{1}{2}, \gamma, \mu}\big)$ being its eigenspaces, see \cite[Theorem~4.6]{Xu21}. The associated eigenvalues are $-n(n + 2\gamma + 2\mu + d - 2)$.\footnote{In \cite[Theorem~4.6, (4.18), p.~3624]{Xu21} the eigenvalue is printed as $-n(n + 2\gamma + 2\mu + d)$. However, the proof given there yields $-n(n + 2\gamma + 2\mu + d - 2)$: the radial factor $g$ of an element of $\mathcal{V}^O_n$ satisfies the equation displayed in \cite[p.~3626]{Xu21}, whose right-hand side is $-(n-m)(n + m + 2\beta + 2\gamma + 2\mu + d - 1)g$, and the subsequent identity has the leading term $-n(n + 2\beta + 2\gamma + 2\mu + d - 1)u$, which for $\beta = -\tfrac{1}{2}$ equals $-n(n + 2\gamma + 2\mu + d - 2)u$. This value is confirmed by applying the operator displayed in \cite[(4.18)]{Xu21} to $u = t \in \mathcal{V}^O_1$, which gives $-(2\gamma + 2\mu + d - 1)t$. The formula for $\mathcal{D}^O_{\gamma, \mu}$ is not used in this paper and we do not reproduce it.}

If $\beta = \frac{1}{2}$ and $\gamma, \mu \geq 0$, then the reproducing kernel of $\mathcal{V}^E_n\big(w_{\beta=\frac{1}{2}, \gamma, \mu}\big)$ is given by the following formula (cf. \cite[Corollary 5.6]{Xu21})
\begin{equation}
    \label{RK2}
    P^E_n\big(w_{\beta=\frac{1}{2}, \gamma, \mu}; (x, t), (y, s)\big) \coloneqq \int \limits_{-1}^{1} \int \limits_{-1}^{1} Z_{n}^{\gamma + \mu + \frac{d}{2}} \big(\xi(x, t, y, s; u, v)\big) \dint\Pi_{\mu - \frac{1}{2}}(u) \dint\Pi_{\gamma - \frac{1}{2}}(v),
\end{equation}
where $\dint\Pi_{\eta}(w)$ is as before, and
\begin{equation*}
    \xi(x, t, y, s; u, v) \coloneqq \xi(u, v) \coloneqq I_1 + vI_2 + uI_3
\end{equation*}
with $I_1$ and $I_2$ as before and $I_3 \coloneqq \sqrt{t^2 - \norm{x}^2} \sqrt{s^2 - \norm{y}^2}$ (cf. the footnote following \eqref{RK1}).
In \eqref{RK2}, we use the special function defined in \eqref{Gegen}.

The relationship between $P^O_n$ and $P^E_n$ is given by the following formula
\begin{equation}
    \label{kernel-relation-solid} P^O_n\big(w_{\beta, \gamma, \mu}; (x, t), (y, s)\big) = \frac{\beta + \mu + \frac{d-1}{2} + \gamma + 1}{\beta + \mu + \frac{d}{2}} st P^E_{n-1}\big(w_{\beta+1, \gamma, \mu}; (x, t), (y, s)\big),
\end{equation}
for $n \in \{1, 2, \ldots\}$, $\beta > -\mu - \frac{d}{2}$ and $\gamma, \mu > -\frac{1}{2}$, see \cite[Theorem~5.5]{Xu21}\footnote{Again, the range $\beta > -\frac{1}{2}$, as specified in \cite[Theorem~5.5]{Xu21}, can be extended as long as the normalizing constants $c_{\beta, \gamma, \mu}$ and $c_{\beta + 1, \gamma, \mu}$ are finite, i.e. to $\beta > -\mu - \tfrac{d}{2}$; this covers $\beta = -\tfrac{1}{2}$ for all $d \ge 2$ and $\mu > -\tfrac{1}{2}$.}; as before, $\mathcal{V}^O_n(w_{\beta, \gamma, \mu}) = t \, \mathcal{V}^E_{n-1}(w_{\beta+1, \gamma, \mu})$.

Similarly to the previous subsection, it is sufficient to find the genuinely sharp estimate for the even kernel. The estimate for the odd kernel is a~consequence of \eqref{kernel-relation-solid}, see Lemma~\ref{odd-even}.

\begin{definition}
	\label{HK2}
	Let $\gamma \in (-\frac{1}{2}, \infty)$ and $\mu \in (-\frac{1}{2}, \infty)$. Then for each $\tau \in (0, \infty)$, the associated \emph{even Jacobi heat kernel} on $\HI$ is given by
    \begin{equation*}
        h^E_\tau\big(w_{\beta=\frac{1}{2}, \gamma, \mu}; (x, t), (y, s)\big) \coloneqq \sum \limits_{n=0}^{\infty} \me^{-\tau n (n + 2\gamma + 2\mu + d)} P^E_n\big(w_{\beta=\frac{1}{2}, \gamma, \mu}; (x, t), (y, s)\big),
    \end{equation*}
    and the associated \emph{odd Jacobi heat kernel} on $\HI$ is given by
    \begin{equation*}
        h^O_\tau\big(w_{\beta=-\frac{1}{2}, \gamma, \mu}; (x, t), (y, s)\big) \coloneqq \sum \limits_{n=1}^{\infty} \me^{-\tau n (n + 2\gamma + 2\mu + d - 2)} P^O_n\big(w_{\beta=-\frac{1}{2}, \gamma, \mu}; (x, t), (y, s)\big).
    \end{equation*}
\end{definition}

\noindent According to this, for the even Jacobi heat kernel on $\HI$ we shall show the following result.

\begin{theorem} \label{T2}
	Let $\gamma \in [0, \infty)$ and $\mu \in [0, \infty)$. Then $h^E_\tau\big(w_{\beta=\frac{1}{2}, \gamma, \mu}; (x, t), (y, s)\big)$, the even Jacobi heat kernel on $\HI$, is comparable to
    \begin{align*}
        &\tau^{- \frac{d}{2} - \frac{1}{2}} \big( \tau + \pi - \arccos(I_1 + I_2 + I_3)\big)^{-\gamma - \mu - \frac{d}{2}} \bigg( \tau + \frac{I_2}{\pi - \arccos(I_1 + I_2 + I_3)} \bigg)^{-\gamma} \\
        &\qquad \times \bigg( \tau + \frac{I_3}{\pi - \arccos(I_1 + I_2 + I_3)} \bigg)^{-\mu} \exp \bigg\{ -\frac{\arccos^2(I_1 + I_2 + I_3)}{4\tau} \bigg\}
    \end{align*}
    for $\tau \in (0, 1]$, and is comparable to 1 for $\tau \in (1, \infty)$, uniformly in $(x, t), (y, s) \in \HI$ (if $I_1 + I_2 + I_3 = -1$, then $I_2 = I_3 = 0$, see Remark~\ref{Remark: range-xi}).
\end{theorem}

A consequence of \eqref{kernel-relation-solid} and Theorem \ref{T2} is the following (see Lemma~\ref{odd-even} and Section~\ref{sec:proofs}).

\begin{corollary}
    \label{C12}
	Let $\gamma \in [0, \infty)$ and $\mu \in [0, \infty)$. Then $h^O_\tau\big(w_{\beta=-\frac{1}{2}, \gamma, \mu}; (x, t), (y, s)\big)$, the odd Jacobi heat kernel on $\HI$, is of the form $st \, k_\tau\big((x, t), (y, s)\big)$, where $k_\tau\big((x, t), (y, s)\big)$ is positive and comparable to
    \begin{align*}
        &\tau^{- \frac{d}{2} - \frac{1}{2}} \big( \tau + \pi - \arccos(I_1 + I_2 + I_3)\big)^{-\gamma - \mu - \frac{d}{2}} \bigg( \tau + \frac{I_2}{\pi - \arccos(I_1 + I_2 + I_3)} \bigg)^{-\gamma} \\
        &\qquad \times \bigg( \tau + \frac{I_3}{\pi - \arccos(I_1 + I_2 + I_3)} \bigg)^{-\mu} \exp \bigg\{ -\frac{\arccos^2(I_1 + I_2 + I_3)}{4\tau} \bigg\}
    \end{align*}
    for $\tau \in (0, 1]$, and to $\me^{-\tau(2\gamma + 2\mu + d - 1)}$ for $\tau \in (1, \infty)$, uniformly in $(x, t), (y, s) \in \HI$. In fact, $k_\tau = c \, \me^{-\tau(2\gamma + 2\mu + d - 1)} h^E_\tau(w_{\beta=\frac{1}{2}, \gamma, \mu}; \cdot, \cdot)$ with the constant $c = c_{d, \gamma, \mu} > 0$ from Lemma~\ref{odd-even}.
\end{corollary}

\subsection{Jacobi Heat Kernel on the Surface of the Hyperboloid.}

The content of this section is derived from \cite[Sections~3--5]{Xu21}. Fix $d \in \{2, 3, \ldots\}$ and consider the domain
\begin{equation}
    \label{def3}
    \prescript{}{\varrho}{\HS} \coloneqq \big\{(x, t) \in \mathbb{R}^{d} \times \mathbb{R} :  \norm{x}^2 = t^2 - \varrho^2, \, \varrho \le \abs{t} \le \sqrt{\varrho^2 + 1} \big\}, \qquad \varrho > 0.
\end{equation}

We equip \eqref{def3} with the weight function
\begin{equation*}
    w_{\beta, \gamma}^{\varrho}(t) = \abs{t}(t^2 - \varrho^2)^{\beta - \frac{1}{2}}(\varrho^2 + 1 - t^2)^{\gamma - \frac{1}{2}}, \qquad \beta, \gamma > -\frac{1}{2}, \quad \varrho > 0.
\end{equation*}

For each $n \in \{0, 1, 2, \dots \}$, the space $\mathcal{V}_n(w_{\beta, \gamma}^{\varrho})$ of orthogonal polynomials of degree $n$ associated with the weight $w_{\beta, \gamma}^{\varrho}$ is defined in terms of the so-called Jacobi polynomials on the surface of the~hyperboloid.

Again, the space $\mathcal{V}_n(w_{\beta, \gamma}^{\varrho})$ can be represented as a direct sum of the even and odd part, but in this section we shall consider only the even part. The subspace containing polynomials odd in the variable $t$ is more complicated -- there is no associated second-order differential operator (see the discussion after the proof of \cite[Theorem~4.4]{Xu21}) and no closed-form expression for the reproducing kernel is available (see \cite[Subsection~5.1.2]{Xu21}).

If $\beta = 0$ and $\gamma > -\frac{1}{2}$, then there exists a second-order differential operator $\mathcal{D}^{E, \varrho}_\gamma$ acting on a~suitable subspace of $\mathcal L^2(\HS, c_{\beta=0, \gamma}^{\varrho} w_{\beta=0, \gamma}^{\varrho})$, where $c_{\beta=0, \gamma}^{\varrho}$ is the normalizing constant, with all $\mathcal{V}^E_n(w_{\beta=0, \gamma}^{\varrho})$ being its eigenspaces. The associated eigenvalues are $-n(n + 2\gamma + d - 1)$.

We do not use the formula for $\mathcal{D}^{E, \varrho}_{\gamma}$, but for completeness we recall that it is equal to
\begin{gather*}
(1 + \varrho^2 - t^2)\bigg(1 - \frac{\varrho^2}{t^2}\bigg)\partial^2_t + \bigg((1 + \varrho^2 - t^2)\frac{\varrho^2}{t^2} - (2\gamma + d)(t^2 - \varrho^2) + d - 1\bigg)\frac{1}{t}\partial_t + \frac{1}{t^2 - \varrho^2} \Delta^{(x)}_0
\end{gather*}
with $\Delta^{(x)}_0$ as before, see \cite[Theorem~4.4]{Xu21}.

Define $t' \coloneqq \sqrt{t^2 - \varrho^2}$ and $s' \coloneqq \sqrt{s^2 - \varrho^2}$. If $\beta = 0$ and  $\gamma \ge 0$, then the reproducing kernel of $\mathcal{V}^E_n(w_{\beta=0, \gamma}^{\varrho})$ is given by the following formula (cf. \cite[(5.11)]{Xu21})
\begin{equation}
    \label{RK3}
    \prescript{}{\varrho}{P^E_n}\big(w_{\beta=0, \gamma}^{\varrho}; (x, t), (y, s)\big) \coloneqq \int \limits_{-1}^{1} Z_{n}^{\gamma + \frac{d-1}{2}} \big(\xi(x, t', y, s'; v)\big) \dint\Pi_{\gamma - \frac{1}{2}}(v).
\end{equation}

\begin{definition}
	\label{HK3}
	Let $\gamma \in [0, \infty)$. Then for each $\tau \in (0, \infty)$, the associated \emph{even Jacobi heat kernel} on $\prescript{}{\varrho}{\HS}$ is given by
    \begin{equation*}
        \prescript{}{\varrho}{h}^E_\tau\big(w_{\beta=0, \gamma}^{\varrho}; (x, t), (y, s)\big) \coloneqq \sum \limits_{n=0}^{\infty} \me^{-\tau n (n + 2\gamma + d - 1)} \prescript{}{\varrho}{P^E_n}\big(w_{\beta=0, \gamma}^{\varrho}; (x, t), (y, s)\big).
    \end{equation*}
\end{definition}

\noindent By \eqref{RK3} the reproducing kernel on the surface of the hyperboloid is the reproducing kernel on the surface of the double cone with the change of variables $t \mapsto \sqrt{t^2 - \varrho^2}$ and $s \mapsto \sqrt{s^2 - \varrho^2}$. Using this, and denoting $I_2^{\varrho} \coloneqq \sqrt{1 + \varrho^2 - t^2}\sqrt{1 + \varrho^2 - s^2}$, we conclude the following.

\begin{corollary}
    \label{C1}
    Let $\gamma \in [0, \infty)$. Then $\prescript{}{\varrho}{h}^E_\tau\big(w_{\beta=0, \gamma}^{\varrho}; (x, t), (y, s)\big)$, the even Jacobi heat kernel on $\prescript{}{\varrho}{\HS}$, is comparable to
    \begin{equation*}
        \tau^{-\frac{d}{2}} \big( \pi - \arccos(I_1 + I_2^{\varrho}) + \tau \big)^{-\gamma - \frac{d}{2} + \frac{1}{2}} \bigg( \tau + \frac{I_2^{\varrho}}{\pi - \arccos(I_1 + I_2^{\varrho})} \bigg)^{-\gamma} \exp \bigg\{-\frac{\arccos^2(
        I_1 + I_2^{\varrho})}{4\tau}\bigg\}
    \end{equation*}
    for $\tau \in (0, 1]$, and is comparable to 1 for $\tau \in (1, \infty)$, uniformly in $(x, t), (y, s) \in \prescript{}{\varrho}{\HS}$.
\end{corollary}

\subsection{Jacobi Heat Kernel on the Solid Hyperboloid.}

The content of this section is derived from \cite[Sections~3--5]{Xu21}. Fix $d \in \{2, 3, \ldots\}$ and consider the domain
\begin{equation}
    \label{def4}
    \prescript{}{\varrho}{\HI} \coloneqq \big\{(x, t) \in \mathbb{R}^{d} \times \mathbb{R} :  \norm{x}^2 \le t^2 - \varrho^2, \, \varrho \le \abs{t} \le \sqrt{\varrho^2 + 1} \big\}, \qquad \varrho > 0.
\end{equation}

We equip \eqref{def4} with the weight function
\begin{equation*}
    w_{\beta, \gamma, \mu}^{\varrho}(t) = \abs{t}(t^2 - \varrho^2)^{\beta - \frac{1}{2}}(\varrho^2 + 1 - t^2)^{\gamma - \frac{1}{2}}(t^2 - \varrho^2 - \norm{x}^2)^{\mu - \frac{1}{2}}, \qquad \beta, \gamma, \mu > -\frac{1}{2}, \quad \varrho > 0.
\end{equation*}

For each $n \in \{0, 1, 2, \dots \}$, the space $\mathcal{V}_n(w_{\beta, \gamma, \mu}^{\varrho})$ of orthogonal polynomials of degree $n$ associated with the weight $w_{\beta, \gamma, \mu}^{\varrho}$ is defined in terms of the so-called Jacobi polynomials on the solid~hyperboloid.

The space $\mathcal{V}_n(w_{\beta, \gamma, \mu}^{\varrho})$ can be represented as a direct sum of the even and odd part, but we shall consider only the even part, as only in this case a suitable second-order differential operator and a closed-form expression for a reproducing kernel are available.

If $\beta = \frac{1}{2}$ and $\gamma, \mu > -\frac{1}{2}$, then there exists a second-order differential operator $\mathcal{D}^{E, \varrho}_{\gamma, \mu}$ acting on a suitable subspace of $\mathcal L^2(\HI, c_{\beta=\frac{1}{2}, \gamma, \mu}^{\varrho} w_{\beta=\frac{1}{2}, \gamma, \mu}^{\varrho})$, where $c_{\beta=\frac{1}{2}, \gamma, \mu}^{\varrho}$ is the normalizing constant, and all $\mathcal{V}^E_n(w_{\beta=\frac{1}{2}, \gamma, \mu}^{\varrho})$ are its eigenspaces. The associated eigenvalues are $-n(n + 2\gamma + 2\mu + d)$.

We do not use the formula for $\mathcal{D}^{E, \varrho}_{\gamma, \mu}$, but for completeness we recall that it is equal to
\begin{gather*}
(1 + \varrho^2 - t^2)\bigg(1 - \frac{\varrho^2}{t^2}\bigg)\partial^2_t + \Delta_x - \inner{x, \nabla_x}^2 + \inner{x, \nabla_x} + \frac{2}{t}(1 + \varrho^2 - t^2) \inner{x, \nabla_x}\partial_t \\
+ \bigg((1 + \varrho^2 - t^2)\frac{\varrho^2}{t^2} + 2\mu + d\bigg)\frac{1}{t} \partial_t - (2\gamma + 2\mu + d + 1) \Bigg(\bigg(1 - \frac{\varrho^2}{t^2}\bigg)t \partial_t + \inner{x, \nabla_x}\Bigg)
\end{gather*}
with $\Delta_x$ and $\nabla_x$ as before, see \cite[Theorem~4.8]{Xu21}.

As before, let $t' \coloneqq \sqrt{t^2 - \varrho^2}$ and $s' \coloneqq \sqrt{s^2 - \varrho^2}$. If $\beta = \frac{1}{2}$ and $\gamma, \mu \ge 0$, then the reproducing kernel of $\mathcal{V}^E_n(w_{\beta=\frac{1}{2}, \gamma, \mu}^{\varrho})$ is given by the following formula (cf. \cite[(5.16)]{Xu21})
\begin{equation}
    \label{RK4}
    \prescript{}{\varrho}{P^E_n}\big(w_{\beta=\frac{1}{2}, \gamma, \mu}^{\varrho}; (x, t), (y, s)\big) \coloneqq \int \limits_{-1}^{1} \int \limits_{-1}^{1} Z_{n}^{\gamma + \mu + \frac{d}{2}} \big(\xi(x, t', y, s'; u, v)\big) \dint\Pi_{\mu - \frac{1}{2}}(u) \dint\Pi_{\gamma - \frac{1}{2}}(v).
\end{equation}

\begin{definition}
	\label{HK4}
	Let $\gamma \in [0, \infty)$ and $\mu \in (-\frac{1}{2}, \infty)$. Then for each $\tau \in (0, \infty)$, the associated \emph{even Jacobi heat kernel} on $\prescript{}{\varrho}{\HI}$ is given by
    \begin{equation*}
        \prescript{}{\varrho}{h}^E_\tau\big(w_{\beta=\frac{1}{2}, \gamma, \mu}^{\varrho}; (x, t), (y, s)\big) \coloneqq \sum \limits_{n=0}^{\infty} \me^{-\tau\, n(n + 2\gamma + 2\mu + d)} \prescript{}{\varrho}{P^E_n}\big(w_{\beta=\frac{1}{2}, \gamma, \mu}^{\varrho}; (x, t), (y, s)\big).
    \end{equation*}
\end{definition}

\noindent By \eqref{RK4} the reproducing kernel on the solid hyperboloid is the reproducing kernel on the solid double cone with the change of variables $t \mapsto \sqrt{t^2 - \varrho^2}$ and $s \mapsto \sqrt{s^2 - \varrho^2}$. Denoting $I_2^{\varrho}$ as before and $I_3^{\varrho} \coloneqq \sqrt{t^2 - \varrho^2 - \norm{x}^2} \sqrt{s^2 - \varrho^2 - \norm{y}^2}$, we conclude the following.

\begin{corollary}
    \label{C2}
	Let $\gamma \in [0, \infty)$ and $\mu \in [0, \infty)$. Then $\prescript{}{\varrho}{h}^E_\tau\big(w_{\beta=\frac{1}{2}, \gamma, \mu}^{\varrho}; (x, t), (y, s)\big)$, the even Jacobi heat kernel on $\prescript{}{\varrho}{\HI}$, is comparable to
    \begin{align*}
        &\tau^{- \frac{d}{2} - \frac{1}{2}} \big( \tau + \pi - \arccos(I_1 + I_2^{\varrho} + I_3^{\varrho})\big)^{-\gamma - \mu - \frac{d}{2}} \bigg( \tau + \frac{I_2^{\varrho}}{\pi - \arccos(I_1 + I_2^{\varrho} + I_3^{\varrho})} \bigg)^{-\gamma} \\
        &\qquad \times \bigg( \tau + \frac{I_3^{\varrho}}{\pi - \arccos(I_1 + I_2^{\varrho} + I_3^{\varrho})} \bigg)^{-\mu} \exp \bigg\{ -\frac{\arccos^2(I_1 + I_2^{\varrho} + I_3^{\varrho})}{4\tau} \bigg\}
    \end{align*}
    for $\tau \in (0, 1]$, and is comparable to 1 for $\tau \in (1, \infty)$, uniformly in $(x, t), (y, s) \in \prescript{}{\varrho}{\HI}$.
\end{corollary}

\subsection{Jacobi Heat Kernel on the Surface of the Paraboloid.}

The content of this section is derived from \cite[Section~3]{Xu23}. Fix $d \in \{2, 3, \ldots\}$ and consider the domain
\begin{equation}
    \label{def5}
    \widetilde{{\mathbb{V}}}_{0}^{d + 1} \coloneqq \big\{(x, t) \in \mathbb{R}^{d} \times \mathbb{R} : \norm{x}^2 = t, \, 0 \le t \le 1 \big\}.
\end{equation}

We equip \eqref{def5} with the weight function
\begin{equation*}
    \widetilde{w}_{\beta, \gamma} = t^\beta(1 - t)^\gamma, \qquad \beta > -\frac{d+1}{2}, \quad \gamma > -1.
\end{equation*}

For each $n \in \{0, 1, 2, \ldots \}$, we define the space $\mathcal{V}_n(\widetilde{w}_{\beta, \gamma})$ of orthogonal polynomials of degree $n$ on the surface $\widetilde{\mathbb{V}}_{0}^{d + 1}$. Each element of this space can be written in the form
\begin{equation*}
    Q^n_{m, l}(x, t) = P^{\beta + m + \frac{d-1}{2}, \gamma}_{n-m}(1 -2t)t^{\frac{m}{2}} Y_l^m\bigg(\frac{x}{\sqrt{t}}\bigg), \qquad 0 \le m \le n, \quad 1 \le l \le L_m^d,
\end{equation*}
where $P^{\alpha, \beta}_n$ is a Jacobi polynomial, $L_m^d \coloneqq \binom{m + d - 1}{m} - \binom{m + d - 3}{m - 2}$, and $\{ Y_l^m: 1 \le l \le L_m^d \}$ spans the space of spherical harmonics of degree $m$ in $d$ variables, see \cite[Proposition~3.1]{Xu23}.

Moreover, there exists a second-order differential operator $\widetilde{D}_\gamma$ acting on an appropriate subspace of $\mathcal L^2(\widetilde{{\mathbb{V}}}_{0}^{d + 1}, \widetilde{c}_{\beta=-\frac{1}{2}, \gamma} \widetilde{w}_{\beta=-\frac{1}{2}, \gamma})$, where $\widetilde{c}_{\beta=-\frac{1}{2}, \gamma}$ is the corresponding normalizing constant. The associated eigenvalues are $-n (n + \gamma + \frac{d}{2}) + m(n + \frac{\gamma + d - 1}{2})$.

For the sake of completeness we recall that $\widetilde{D}_\gamma$ is given by
\begin{equation*}
    t(1 - t) \partial_t^2 + \frac{d}{2} \partial_t - \bigg(\gamma + \frac{d}{2} + 1 \bigg) t \partial_t + \frac{1 - t}{4t} \Delta_0^{(\xi)},
\end{equation*}
where $\Delta_0^{(\xi)}$ denotes the Laplace–Beltrami operator acting in the direction of the spherical variable $\xi \coloneqq \frac{x}{\sqrt{t}}$, see \cite[Proposition~3.2]{Xu23}.

\begin{observation}
    \label{O1}
    We cannot establish sharp estimates for Jacobi heat kernel on $\widetilde{{\mathbb{V}}}_{0}^{d + 1}$ related to $\widetilde{D}_\gamma$ using known techniques presented in this paper. This is due to the fact that the associated eigenvalues depend on both $n$ and $m$, hence $\mathcal{V}_n(\widetilde{w}_{\beta=-\frac{1}{2}, \gamma})$ are not eigenspaces of $\widetilde{D}_\gamma$.
\end{observation}

\subsection{Jacobi Heat Kernel on the Solid Paraboloid.}

The content of this section is derived from \cite[Section~4]{Xu23}. Fix $d \in \{2, 3, \ldots\}$ and consider the domain
\begin{equation}
    \label{def6}
    \widetilde{\mathbb{V}}^{d + 1} \coloneqq \big\{(x, t) \in \mathbb{R}^{d} \times \mathbb{R} : \norm{x}^2 \le t, \, 0 \le t \le 1 \big\}.
\end{equation}

We equip \eqref{def6} with the weight function
\begin{equation*}
    \widetilde{w}_{\beta, \gamma, \mu} = t^\beta(1 - t)^\gamma(t - \norm{x}^2)^{\mu - \frac{1}{2}}, \qquad \beta > -\frac{d+1}{2}, \quad \gamma > -1, \quad \mu > -\frac{1}{2}.
\end{equation*}

For each $n \in \{0, 1, 2, \ldots \}$, we define the space $\mathcal{V}_n(\widetilde{w}_{\beta, \gamma, \mu})$ of orthogonal polynomials of degree $n$ on the solid $\widetilde{\mathbb{V}}^{d + 1}$. Each element of this space can be written in the form
\begin{equation*}
    Q^n_{m, k}(x, t) = P^{\beta + \mu + m + \frac{d-1}{2}, \gamma}_{n-m}(1 -2t)t^{\frac{m}{2}} P^m_{\bf k} \bigg(\frac{x}{\sqrt{t}}\bigg), \qquad 0 \le m \le n, \quad \abs{{\bf{k}}} = m,
\end{equation*}
where $P^{\alpha, \beta}_n$ is a Jacobi polynomial and $\{ P^m_{\bf k}: {\bf k} \in \{0, 1, 2, \ldots\}^d, \, \abs{{\bf k}} = m \}$ is the so-called orthonormal basis with parity related to the $d$-dimensional ball, see \cite[Proposition~4.1]{Xu23}.

Moreover, there exists a second-order differential operator $\widetilde{D}_{\gamma, \mu}$ acting on an appropriate subspace of $\mathcal L^2(\widetilde{\mathbb{V}}^{d + 1}, \widetilde{c}_{\beta=0, \gamma, \mu} \widetilde{w}_{\beta=0, \gamma, \mu})$, where $\widetilde{c}_{\beta=0, \gamma, \mu}$ is the corresponding normalizing constant. The associated eigenvalues are $-n (n + \mu + \gamma + \frac{d+1}{2}) + m(n + \mu + \frac{\gamma + d}{2})$.

For the sake of completeness we recall that $\widetilde{D}_{\gamma, \mu}$ is given by
\begin{align*}
     t(1 - t) \partial_t^2 + (1 - t) \langle x, \nabla_x \rangle \partial_t + \frac{1}{4} (1 - t) \Delta_x + \bigg(\mu + \frac{d + 1}{2} \bigg) (1 - t) \partial_t - \frac{\gamma + 1}{2} \bigg(2t \partial_t + \langle x, \nabla_x \rangle\bigg)
\end{align*}
with $\Delta_x$ and $\nabla_x$ as before, see \cite[Proposition~4.2]{Xu23}.

\begin{observation}
    \label{O2}
    We cannot establish sharp estimates for Jacobi heat kernel on $\widetilde{\mathbb{V}}^{d + 1}$ related to $\widetilde{D}_{\gamma, \mu}$ using known techniques presented in this paper. This is due to the fact that the associated eigenvalues depend on both $n$ and $m$, hence $\mathcal{V}_n(\widetilde{w}_{\beta=0, \gamma, \mu})$ are not eigenspaces of $\widetilde{D}_{\gamma, \mu}$.
\end{observation}

\subsection{Summary of the results.}

\begin{center}
    \begin{table}[ht]
    \captionsetup{font=small}
    \caption{Summary of the obtained results.}
    \renewcommand{\arraystretch}{1.5}
    \begin{tabular}{c || c c }
        \label{Table1}
         Space & Even Jacobi Heat Kernel & Odd Jacobi Heat Kernel \\
         \hline
         \hline
         $\HS$ & Available for $\beta=0$ & Available for $\beta=-1$ ($d \ge 3$) \\
         \hline
         $\HI$ & Available for $\beta=\frac{1}{2}$ & Available for $\beta=-\frac{1}{2}$ \\
         \hline
         $\prescript{}{\varrho}{\HS}$ & Available for $\beta=0$ & Not Available\textsuperscript{1} \\
         \hline
         $\prescript{}{\varrho}{\HI}$ & Available for $\beta=\frac{1}{2}$ & Not Available\textsuperscript{1} \\
         \hline
         $\widetilde{{\mathbb{V}}}_{0}^{d + 1}$ & Not Available\textsuperscript{2} & Not Available\textsuperscript{2} \\
         \hline
         $\widetilde{{\mathbb{V}}}^{d + 1}$ & Not Available\textsuperscript{2} & Not Available\textsuperscript{2} \\ [1ex]
    \end{tabular}
    \vspace{2mm}
    \caption*{\textsuperscript{1}Closed-form formula is not available.\\\textsuperscript{2}Polynomials with a fixed degree do not span eigenspaces.}
    \end{table}
\end{center}

\subsection{Comments} Several remarks are in order.
\begin{enumerate}[label=(\alph*)]
	\item The resulting estimates are genuinely sharp, representing one of the rare cases where this degree of accuracy has been achieved.
	\item Multidimensional double cones and hyperboloids combine geometric features of both intervals and Euclidean balls, crucial for analyzing the corresponding Jacobi framework. For both settings, when $d=1$, the geometry simplifies considerably since the Euclidean component is absent, and thus this scenario is not covered in this work. For more information, in this case, we refer to \cite[Section~2]{Xu21}.
	\item The parameters $\beta$ related to the factors $t^\beta$ in the weights are set to specific values, see Table \ref{Table1}, as only then the subspaces generated by the Jacobi polynomials of degree $n$ are eigenspaces of the underlying differential operators, see \cite[Theorems~4.2~and~4.4]{Xu21}. The parameter $\gamma$ is restricted in accordance to the ranges in which useful, i.e. closed,  formulas for the associated reproducing kernels hold, see \cite[Theorems~5.1~and~5.5]{Xu21}.
	\item We focus mainly on $\tau \in (0, 1]$; the uniform estimates for $\tau \in (1, \infty)$ follow from the known large-time behaviour of the classical Jacobi heat kernel, see Section~\ref{sec:proofs}.
	\item The idea behind Theorems~\ref{T1}~and~\ref{T2} (and thereby all corollaries) is to express the studied Jacobi heat kernels in terms of the heat kernel related to the classical Jacobi expansions on $[-1,1]$, see Lemmas~\ref{HKG1}~and~\ref{HKG2}. In principle, we follow the strategy proposed in \cite{NSS21}.
\end{enumerate}



\section{Technical preparation} In this section, we compile some supporting results that will be instrumental in proving the main theorems. We begin by providing an alternative representation of the Jacobi heat kernels, adapting the approach from \cite{NSS21}. The core objective is to eliminate the oscillatory behavior of Jacobi polynomials, transforming them into a positive expression that is significantly more manageable for analysis.

Given $\lambda \in [0, \infty)$, we recall the formula for the heat kernel $G_{\tau }^{\lambda, \lambda}$ associated with the classical Jacobi polynomials $P_n^{\lambda,\lambda}$ on $[-1,1]$, when the first argument equals $1$. For $\tau \in (0, \infty)$, we have
\begin{equation*}
G_{\tau }^{\lambda, \lambda}(1, w) \coloneqq \sum \limits_{n=0}^{\infty} e^{-\tau  \lambda_{n}} \frac{P_n^{\lambda, \lambda}(1) P_n^{\lambda, \lambda}(w)}{h_n^{\lambda, \lambda}}
= \sum \limits_{n=0}^{\infty} e^{-\tau  \lambda_{n}} Z_n^{\lambda + \frac{1}{2}}(w),
\end{equation*}
where $\lambda_{n} \coloneqq n(n + 2\lambda + 1)$ and $Z_n^{\lambda}$ is defined in \eqref{Gegen}.

The heat kernel from Definition~\ref{HK1} can be described in terms of $G_{\tau }^{\lambda, \lambda}$.

\begin{lemma}
	\label{HKG1}
	For each $\tau \in (0,\infty)$, the even Jacobi heat kernel from Definition~\ref{HK1} can be given by
	\begin{equation*}
	    \int \limits_{-1}^{1} G_\tau^{\gamma + \frac{d}{2} - 1, \gamma + \frac{d}{2} -1}\big(1, \xi(x, t, y, s; v) \big) \measure{\gamma}{v}.
	\end{equation*}
\end{lemma}

\begin{proof}
	By expanding left hand side of the formula from Definition~\ref{HK1} we get
	\begin{equation*}
	\sum \limits_{n=0}^{\infty} \me^{-\tau n (n + 2\gamma + d - 1)} \int \limits_{-1}^{1} Z_{n}^{\gamma + \frac{d-1}{2}} \big(\xi(x, t, y, s; v)\big) \measure{\gamma}{v}.
	\end{equation*}
	Note that $n (n + 2\gamma + d - 1) = n (n + 2 \lambda + 1)$, where $\lambda = \gamma + \frac{d}{2} - 1$. After the application of Fubini's theorem the proof is completed.
\end{proof}

\begin{lemma}
	\label{HKG2}
	For each $\tau \in (0,\infty)$, the even Jacobi heat kernel from Definition~\ref{HK2} can be given by
	\begin{equation*}
	    \int \limits_{-1}^{1} \int \limits_{-1}^{1} G_\tau^{\gamma + \mu + \frac{d}{2} - \frac{1}{2}, \gamma + \mu + \frac{d}{2} - \frac{1}{2}}\big(1, \xi(x, t, y, s; u, v) \big) \measure{\mu}{u} \measure{\gamma}{v}.
	\end{equation*}
\end{lemma}

\begin{proof}
	By expanding left hand side of the formula from Definition~\ref{HK2} we get
	\begin{equation*}
	 \sum \limits_{n=0}^{\infty} \me^{-\tau n (n + 2\gamma + 2\mu + d)} \int \limits_{-1}^{1} \int \limits_{-1}^{1} Z_{n}^{\gamma + \mu + \frac{d}{2}} \big(\xi(x, t, y, s; u, v)\big) \measure{\mu}{u} \measure{\gamma}{v}.
	\end{equation*}
	Observe that $n (n + 2\gamma + 2\mu + d) = n (n + 2 \lambda + 1)$, where $\lambda = \gamma + \mu + \frac{d}{2} - \frac{1}{2}$. Fubini's theorem concludes the proof.
\end{proof}

The rest of this section is devoted to recalling various estimates obtained in previous research in the context of the Jacobi heat kernel on $\mathbb{B}^d$. The first one describes the behavior of the function $G_{\tau}^{\lambda, \lambda}$ (cf.~\cite[(7)]{NSS21}).

\begin{lemma}
	\label{LNSS1}
	Fix $\lambda \in [0, \infty)$. Then
	\begin{equation*}
	G_{\tau}^{\lambda, \lambda}(1, \cos \psi) \simeq \tau^{-\lambda- 1} ( \tau + \pi - \psi )^{-\lambda - \frac{1}{2}} \exp\bigg\{-\frac{\psi^2}{4\tau} \bigg\}
	\end{equation*}
	uniformly in $\tau \in (0, 1]$ and $\psi \in [0, \pi]$.
\end{lemma}

\noindent We shall use the following elementary estimate (cf.~\cite[(9)]{NSS21}).

\begin{lemma}
	\label{LNSS2}
	Fix $\kappa \in \RR$. Then
	\begin{equation*}
	(\tau + \pi - \eta)^{\kappa} \exp \bigg\{ - \frac{\eta^2}{\tau} \bigg\} \lesssim (\tau + \pi - \theta)^{\kappa} \exp \bigg\{ - \frac{\theta^2}{\tau} \bigg\}
	\end{equation*}
	uniformly in $\tau \in (0, \infty)$ and $\theta, \eta \in [0,\pi]$ with $\theta \leq \eta$.
\end{lemma}

\noindent Moreover the following integral estimate will be required in the proof (cf. \cite[Lemma 2.1]{NSS21}).

\begin{lemma}
	\label{LNSS3}
	Fix $\nu \in [-\frac{1}{2}, \infty)$, $\eta \in \mathbb{R}$, and let $\Phi_{A, B}(w) \coloneqq \arccos(A + Bw)$. Then
	\begin{align*}
	&\int \limits_{0}^{1} \big(\pi - \Phi_{A, B}(w) + D\big)^{-\eta} \exp \bigg\{ -\frac{\Phi_{A, B}^2(w)}{D} \bigg\} \dint \Pi_{\nu}(w) \\
	&\qquad \simeq D^{\nu + \frac{1}{2}} \big(\pi - \Phi_{A, B}(1) + D\big)^{-\eta} \bigg( B\big(\pi - \Phi_{A, B}(1)\big)^{-1} + D \bigg)^{-\nu - \frac{1}{2}} \exp \bigg\{ -\frac{\Phi_{A, B}^2(1)}{D} \bigg\}.
	\end{align*}
	uniformly in $B \in [0,1]$, $A \in [-1, 1-B]$, and $D \in (0, \infty)$ (with $B (\pi -  \Phi_{A, B}(1))^{-1} = 0$ if $B=0$).
\end{lemma}

\begin{remark}
    \label{remark1}
    In fact, Lemma \ref{LNSS3} can be generalized to
    \begin{align*}
	&\int \limits_{0}^{1} \big(\pi - \Phi_{A, B}(w) + cD\big)^{-\eta} \exp \bigg\{ -\frac{\Phi_{A, B}^2(w)}{D} \bigg\} \dint \Pi_{\nu}(w) \\
	&\qquad \simeq D^{\nu + \frac{1}{2}} \big(\pi - \Phi_{A, B}(1) + D\big)^{-\eta} \bigg( B\big(\pi - \Phi_{A, B}(1)\big)^{-1} + D \bigg)^{-\nu - \frac{1}{2}} \exp \bigg\{ -\frac{\Phi_{A, B}^2(1)}{D} \bigg\}.
	\end{align*}
    for any fixed $c > 0$. This observation is a straightforward consequence of the fact that for any fixed $c > 0$ and $\alpha \in \RR$, we have $(x + y)^\alpha \simeq (x + cy)^\alpha$ in the range $(x, y) \in [0, \infty) \times (0, \infty)$.
\end{remark}

In the proof of Theorem~\ref{T2} the integrand will contain an additional factor depending on the variable of integration. The following extension of Lemma~\ref{LNSS3} takes care of it.

\begin{lemma}
	\label{LNSS3ext}
	Fix $\nu \in [-\frac{1}{2}, \infty)$, $\eta \in \RR$, $\mu \in [0, \infty)$, and let $\Phi_{A, B}$ be as in Lemma~\ref{LNSS3}. Then
	\begin{align*}
	&\int \limits_{0}^{1} \big(\pi - \Phi_{A, B}(w) + D\big)^{-\eta} \bigg( B'\big(\pi - \Phi_{A, B}(w)\big)^{-1} + D \bigg)^{-\mu} \exp \bigg\{ -\frac{\Phi_{A, B}^2(w)}{D} \bigg\} \dint \Pi_{\nu}(w) \\
	&\qquad \simeq \bigg( B'\big(\pi - \Phi_{A, B}(1)\big)^{-1} + D \bigg)^{-\mu} D^{\nu + \frac{1}{2}} \big(\pi - \Phi_{A, B}(1) + D\big)^{-\eta} \\
	&\qquad \qquad \times \bigg( B\big(\pi - \Phi_{A, B}(1)\big)^{-1} + D \bigg)^{-\nu - \frac{1}{2}} \exp \bigg\{ -\frac{\Phi_{A, B}^2(1)}{D} \bigg\}
	\end{align*}
	uniformly in $B, B' \in [0, 1]$, $A \in [-1, 1 - B]$ with $A > -1$ if $B' > 0$, and $D \in (0, \infty)$ (with the convention $B' (\pi - \Phi_{A, B}(w))^{-1} = 0$ if $B' = 0$). Moreover, as in Remark~\ref{remark1}, $D$ may be replaced by $cD$, $c > 0$ fixed, in the first two factors of the integrand.
\end{lemma}

\begin{proof}
	Write $\psi(w) \coloneqq \pi - \Phi_{A, B}(w)$, $\delta \coloneqq \psi(1)$, $g(w) \coloneqq \big( B' \psi(w)^{-1} + D \big)^{-\mu}$, and let $F$ be the integrand of Lemma~\ref{LNSS3}. Since $\psi$ is nondecreasing and $\mu \ge 0$, $g$ is nondecreasing, so $g \le g(1)$ and the upper bound follows from Lemma~\ref{LNSS3}. For the lower bound we may assume $B' > 0$, so that $A > -1$ and $\delta > 0$. We have $\cos \psi(w) = -(A + Bw) = \cos \delta + B(1 - w)$ and $B \le 1 + A + B = 1 - \cos \delta$, hence for $w \in [\frac{1}{2}, 1]$
	\begin{equation*}
	    \cos \psi(w) \le \cos \delta + \tfrac{1}{2}(1 - \cos \delta) = \cos^2 \tfrac{\delta}{2} \le \cos \tfrac{\delta}{2},
	\end{equation*}
	so that $\psi(w) \ge \frac{\delta}{2}$ and $g(w) \ge (2 B' \delta^{-1} + D)^{-\mu} \ge 2^{-\mu} g(1)$. By Lemma~\ref{LNSS2} (with $\tau = D$, $\kappa = -\eta$, $\theta = \Phi_{A, B}(w')$, and $\eta$ there replaced by $\Phi_{A, B}(w)$) we have $F(w) \lesssim F(w')$ whenever $w \le w'$. Averaging over $w' \in [\frac{1}{2}, 1]$ with respect to $\dint \Pi_\nu$ (note that $\Pi_\nu([\frac{1}{2}, 1]) > 0$) gives $\int_0^{1/2} F \dint \Pi_\nu \lesssim \int_{1/2}^1 F \dint \Pi_\nu$, hence $\int_{1/2}^1 F \dint \Pi_\nu \gtrsim \int_0^1 F \dint \Pi_\nu$. Therefore the left-hand side is bounded from below by $2^{-\mu} g(1) \int_{1/2}^1 F \dint \Pi_\nu \gtrsim g(1) \int_0^1 F \dint \Pi_\nu$, and Lemma~\ref{LNSS3} completes the proof. The last assertion follows as in Remark~\ref{remark1}.
\end{proof}

Finally, we express the odd Jacobi heat kernels in terms of the even ones.

\begin{lemma}
	\label{odd-even}
	Let $\gamma, \mu \in (-\frac{1}{2}, \infty)$ and $\tau \in (0, \infty)$, and put $c_{d, \gamma} \coloneqq \frac{\gamma + \frac{d-1}{2}}{\frac{d-2}{2}}$ and $c_{d, \gamma, \mu} \coloneqq \frac{\mu + \gamma + \frac{d}{2}}{\mu + \frac{d-1}{2}}$. If $d \ge 3$, then
	\begin{equation*}
	    h^O_\tau\big(w_{\beta=-1, \gamma}; (x, t), (y, s)\big) = c_{d, \gamma} \, \me^{-\tau(2\gamma + d - 2)} \, st \, h^E_\tau\big(w_{\beta=0, \gamma}; (x, t), (y, s)\big).
	\end{equation*}
	Moreover, for every $d \ge 2$,
	\begin{equation*}
	    h^O_\tau\big(w_{\beta=-\frac{1}{2}, \gamma, \mu}; (x, t), (y, s)\big) = c_{d, \gamma, \mu} \, \me^{-\tau(2\gamma + 2\mu + d - 1)} \, st \, h^E_\tau\big(w_{\beta=\frac{1}{2}, \gamma, \mu}; (x, t), (y, s)\big).
	\end{equation*}
\end{lemma}

\begin{proof}
	The first identity follows by inserting \eqref{kernel-relation-surface} with $\beta = -1$ into Definition~\ref{HK1} and using $n(n + 2\gamma + d - 3) = (n-1)\big((n-1) + 2\gamma + d - 1\big) + (2\gamma + d - 2)$. The second one follows in the same way from \eqref{kernel-relation-solid} with $\beta = -\frac{1}{2}$, Definition~\ref{HK2}, and $n(n + 2\gamma + 2\mu + d - 2) = (n-1)\big((n-1) + 2\gamma + 2\mu + d\big) + (2\gamma + 2\mu + d - 1)$.
\end{proof}

\section{Proofs of the main results} \label{sec:proofs}

We are ready to prove Theorems \ref{T1} and \ref{T2}.

\begin{proof}[Proof of Theorem~\ref{T1}]
    Let us fix $(x, t), (y, s) \in \HS$ and $\tau \in (0, 1]$. We consider the formula for $h^E_\tau\big(w_{\beta=0, \gamma}; (x, t), (y, s)\big)$ from Lemma \ref{HKG1}. Using the estimate from Lemma \ref{LNSS1} we conclude that the integrand is comparable to
    \begin{equation*}
        \tau^{-\gamma - \frac{d}{2}} \big(\tau + \pi - \arccos\xi(v) \big)^{-\gamma - \frac{d}{2} + \frac{1}{2}} \exp \bigg\{-\frac{\arccos^2\xi(v)}{4\tau}\bigg\}.
    \end{equation*}
    This formula is well defined, because $\xi(v) \in [-1, 1]$, as shown in Lemma \ref{Lemma: range-xi} from Appendix.

    Moreover, note that $\xi(v)$ is linear with respect to the variable $v$, and the value of $\xi(v)$ is nondecreasing in $v$ (as $I_2 \ge 0$). This observation, together with Lemma \ref{LNSS2} (applied with $4\tau$ in place of $\tau$, cf. Remark~\ref{remark1}) and the symmetry of $\dint \Pi_{\gamma - \frac{1}{2}}$, leads us to the reduction of integration interval, so the integral is comparable to
    \begin{equation*}
        \tau^{-\gamma - \frac{d}{2}} \int \limits_0^1 \big(\tau + \pi - \arccos\xi(v)\big)^{-\gamma - \frac{d}{2} + \frac{1}{2}} \exp \bigg\{-\frac{\arccos^2\xi(v)}{4\tau} \bigg\} \measure{\gamma}{v}.
    \end{equation*}

    Now we use Lemma \ref{LNSS3} (and Remark \ref{remark1}) with $\eta = \gamma + \frac{d}{2} - \frac{1}{2}$, $A = I_1$, $B = I_2$, $D = 4\tau$, and $\nu = \gamma - \frac{1}{2}$; the hypotheses $B \in [0, 1]$ and $A \in [-1, 1 - B]$ hold, since $I_2 \in [0,1]$ and $-1 \le \xi(-1) \le \xi(0) = I_1$, $I_1 + I_2 = \xi(1) \le 1$ by Lemma~\ref{Lemma: range-xi}. Since $D^{\nu + \frac{1}{2}} = (4\tau)^{\gamma}$, we obtain that $h^E_\tau\big(w_{\beta=0, \gamma}; (x, t), (y, s)\big)$ is comparable to
    \begin{equation*}
        \tau^{-\frac{d}{2}} \big( \tau + \pi - \arccos\xi(1) \big)^{-\gamma - \frac{d}{2} + \frac{1}{2}} \bigg(\tau + \frac{I_2}{\pi - \arccos\xi(1)}\bigg)^{-\gamma} \exp \bigg\{-\frac{\arccos^2\xi(1)}{4\tau} \bigg\},
    \end{equation*}
    which is the asserted bound, as $\xi(1) = I_1 + I_2$.
\end{proof}

\begin{proof}[Proof of Theorem~\ref{T2}]
    Let $(x, t), (y, s) \in \HI$ and $\tau \in (0, 1]$ be fixed. We consider the formula for $h^E_\tau\big(w_{\beta=\frac{1}{2}, \gamma, \mu}; (x, t), (y, s)\big)$ from Lemma \ref{HKG2}. Using the estimate from Lemma \ref{LNSS1} we obtain that the integrand is comparable to
    \begin{equation*}
        \tau^{-\gamma - \mu - \frac{d}{2} - \frac{1}{2}} \big(\tau + \pi - \arccos\xi(u, v)\big)^{-\gamma - \mu - \frac{d}{2}} \exp \bigg\{-\frac{\arccos^2\xi(u, v)}{4\tau}\bigg\}.
    \end{equation*}
    The above formula is well defined, because $\xi(u, v) \in [-1, 1]$, as proved in Lemma \ref{Lemma: range-xi}.

    Observe that $\xi(u, v)$ is nondecreasing in each of the variables $u$ and $v$ (as $I_2, I_3 \ge 0$). This observation, together with Lemma \ref{LNSS2} (applied with $4\tau$ in place of $\tau$, separately in $u$ and in $v$) and the symmetry of the measures, leads us to the reduction of the intervals in the~integral
    \begin{equation*}
      \tau^{-\gamma - \mu - \frac{d}{2} - \frac{1}{2}} \int \limits_{0}^{1} \int \limits_{0}^{1} \big( \tau + \pi - \arccos \xi(u, v) \big)^{-\gamma - \mu - \frac{d}{2}} \exp\bigg\{-\frac{\arccos^2 \xi(u, v)}{4\tau}\bigg\} \measure{\mu}{u} \measure{\gamma}{v}.
    \end{equation*}

    To complete the proof we integrate in $u$ and $v$ separately.

    Let us start with variable $u$. We use Lemma \ref{LNSS3} and Remark \ref{remark1} with $\eta = \gamma + \mu + \frac{d}{2}$, $A = I_1 + vI_2$, $B = I_3$, $D = 4\tau$, and $\nu = \mu - \frac{1}{2}$; the hypotheses are satisfied, since $I_3 \in [0, 1]$ and $-1 \le \xi(-1, v) \le \xi(0, v) = A$, $A + B = \xi(1, v) \le 1$ by Lemma~\ref{Lemma: range-xi}. As $D^{\nu + \frac{1}{2}} = (4\tau)^{\mu}$, we get that
    \begin{align*}
        &h^E_\tau\big(w_{\beta=\frac{1}{2}, \gamma, \mu}; (x, t), (y, s)\big) \simeq \tau^{-\gamma - \frac{d}{2} - \frac{1}{2}} \int \limits_{0}^{1} \big(\tau + \pi - \arccos \xi(1, v)\big)^{-\gamma - \mu - \frac{d}{2}} \\
        &\qquad \times \bigg(\tau + \frac{I_3}{\pi - \arccos\xi(1, v)}\bigg)^{-\mu} \exp \bigg\{-\frac{\arccos^2 \xi(1, v)}{4\tau}\bigg\} \measure{\gamma}{v}.
    \end{align*}

    The second factor of the integrand depends on $v$, so we now use Lemma~\ref{LNSS3ext} for the remaining variable $v$. As a function of $v$, we have $\arccos \xi(1, v) = \Phi_{A, B}(v)$ with $A = I_1 + I_3$ and $B = I_2$, and we put $B' = I_3$, $\eta = \gamma + \mu + \frac{d}{2}$, $D = 4\tau$, and $\nu = \gamma - \frac{1}{2}$ (the same $\mu$ as above). Again $A \in [-1, 1-B]$ by Lemma~\ref{Lemma: range-xi}, and if $I_3 > 0$, then $A = \xi(1, 0) > -1$ by Remark~\ref{Remark: range-xi}; the factor $\tau$ in place of $4\tau$ is covered by the last assertion of Lemma~\ref{LNSS3ext}. Since $D^{\nu + \frac{1}{2}} = (4\tau)^{\gamma}$, we conclude that
    \begin{align*}
        &h^E_\tau\big(w_{\beta=\frac{1}{2}, \gamma, \mu}; (x, t), (y, s)\big) \simeq \tau^{- \frac{d}{2} - \frac{1}{2}} \bigg(\tau + \frac{I_2}{\pi - \arccos\xi(1, 1)} \bigg)^{-\gamma} \bigg(\tau + \frac{I_3}{\pi - \arccos\xi(1, 1)} \bigg)^{-\mu} \\
        &\quad \times  \big(\tau + \pi - \arccos \xi(1, 1)\big)^{-\gamma - \mu - \frac{d}{2}} \exp \bigg\{-\frac{\arccos^2 \xi(1, 1)}{4\tau}\bigg\},
    \end{align*}
    which ends the proof, as $\xi(1, 1) = I_1 + I_2 + I_3$.
\end{proof}

It remains to justify the large-time parts of Theorems~\ref{T1} and \ref{T2}. For $\tau \in [1, \infty)$ we have $G_\tau^{\lambda, \lambda}(1, w) \simeq 1$ uniformly in $w \in [-1, 1]$, see \cite[(5)]{NSS21} and \cite{NS13}. Since $\dint \Pi_{\gamma - \frac{1}{2}}$ and $\dint \Pi_{\mu - \frac{1}{2}}$ are probability measures, Lemmas~\ref{HKG1} and \ref{HKG2} give $h^E_\tau \simeq 1$ for $\tau \in (1, \infty)$, uniformly in the points, in both settings.

\begin{proof}[Proof of Corollaries~\ref{C11} and \ref{C12}]
    By Lemma~\ref{odd-even} we have $h^O_\tau(w_{\beta=-1, \gamma}) = st \, k_\tau$, where the function $k_\tau \coloneqq c_{d, \gamma} \, \me^{-\tau(2\gamma + d - 2)} \, h^E_\tau(w_{\beta=0, \gamma})$ is positive by Theorem~\ref{T1}. For $\tau \in (0, 1]$ we have $\me^{-\tau(2\gamma + d - 2)} \in [\me^{-(2\gamma + d - 2)}, 1]$, so $k_\tau$ is comparable to the bound of Theorem~\ref{T1}, while for $\tau \in (1, \infty)$ the large-time part of Theorem~\ref{T1} gives $k_\tau \simeq \me^{-\tau(2\gamma + d - 2)}$. Corollary~\ref{C12} follows in the same way from Theorem~\ref{T2}.
\end{proof}

Corollaries~\ref{C1} and \ref{C2} follow from Theorems~\ref{T1} and \ref{T2}, since by \eqref{RK3} and \eqref{RK4} (and the equality of the eigenvalues) the heat kernels on the hyperboloids coincide with those on the double cones after the change of variables $t \mapsto \sqrt{t^2 - \varrho^2}$, $s \mapsto \sqrt{s^2 - \varrho^2}$, which does not affect $\inner{x, y}$.

\section{Appendix} \label{Appendix}

\begin{lemma}
    \label{Lemma: range-xi}
    Let $\varrho \ge 0$ and $\varrho \le \abs{s}, \abs{t} \le \sqrt{\varrho^2 + 1}$. Moreover, let $u, v \in [-1, 1]$. We define
    \begin{align*}
        \xi_{\varrho}(x, t, y, s; u, v) &\coloneqq \inner{x, y} + u\sqrt{t^2 - \varrho^2 - \norm{x}^2}\sqrt{s^2 - \varrho^2 - \norm{y}^2} \\
        &\qquad + v\sqrt{1 + \varrho^2 - s^2} \sqrt{1 + \varrho^2 - t^2},
    \end{align*}
    where $\norm{x}^2 \le t^2 - \varrho^2$ and $\norm{y}^2 \le s^2 - \varrho^2$. Then we have $\abs{\xi_{\varrho}(x, t, y, s; u, v)} \le 1$.
\end{lemma}

\begin{proof}
    Observe that $0 \le t^2 - \varrho^2 \le 1$ and $0 \le s^2 - \varrho^2 \le 1$. We define $\alpha \coloneqq \sqrt{t^2 - \varrho^2}$ and $\beta \coloneqq \sqrt{s^2 - \varrho^2}$. We have
    \begin{gather*}
        \inner{x, y} + u\sqrt{t^2 - \varrho^2 - \norm{x}^2}\sqrt{s^2 - \varrho^2 - \norm{y}^2} + v\sqrt{1 + \varrho^2 - s^2} \sqrt{1 + \varrho^2 - t^2} \\
        = \inner{x, y} + u\sqrt{\alpha^2 - \norm{x}^2}\sqrt{\beta^2 - \norm{y}^2} + v\sqrt{1 - \alpha^2}\sqrt{1 - \beta^2}.
    \end{gather*}

    Consider two vectors, $\tilde{x}$ and $\tilde{y}$, defined as follows
    \begin{equation*}
        \tilde{x} \coloneqq \big(x, u\sqrt{\alpha^2 - \norm{x}^2}, v\sqrt{1 - \alpha^2}\big), \qquad \tilde{y} \coloneqq \big(y, \sqrt{\beta^2 - \norm{y}^2}, \sqrt{1 - \beta^2}\big).
    \end{equation*}

    Note that $\inner{\tilde{x}, \tilde{y}} = \xi_{\varrho}(x, t, y, s; u, v)$. Moreover, a straightforward computation shows that $\norm{\tilde{x}}^2 \le \norm{\tilde{y}}^2 = 1$. The application of the Cauchy--Schwarz inequality completes the proof.
\end{proof}

\begin{remark}
    \label{Remark: range-xi}
    In Lemma \ref{Lemma: range-xi}, we can distinguish several special cases:
    \begin{enumerate}
        \item the case $\varrho = 0$ corresponds to the double cone (its surface and the solid domain), whereas the case $\varrho > 0$ relates to the hyperbolic settings,
        \item the special case $\norm{x}^2 = t^2 - \varrho^2$, $\norm{y}^2 = s^2 - \varrho^2$ (where the term containing $u$ vanishes) corresponds to the surfaces, whereas the general case relates to the solid domains,
        \item if $\xi_{\varrho}(x, t, y, s; u, v) = -1$, then $\tilde{x} = -\tilde{y}$ (equality in the Cauchy--Schwarz inequality), so comparing the last two coordinates, $v = 1$ forces $\sqrt{1 - \alpha^2} = \sqrt{1 - \beta^2} = 0$ and $u = 1$ forces $\sqrt{\alpha^2 - \norm{x}^2} = \sqrt{\beta^2 - \norm{y}^2} = 0$. In particular, $I_1 + I_2 = -1$ implies $I_2 = 0$, $I_1 + I_2 + I_3 = -1$ implies $I_2 = I_3 = 0$, and $I_1 + I_3 = -1$ implies $I_3 = 0$ (likewise for $I_2^\varrho$, $I_3^\varrho$).
    \end{enumerate}
\end{remark}

\end{document}